\documentclass[12pt,reqno]{amsart}
\usepackage{fullpage}
\usepackage{mathrsfs}
\usepackage{amssymb}
\usepackage{amsfonts}
\usepackage{amsmath}
\usepackage{graphicx}
\usepackage{hyperref}
\usepackage{float}
\usepackage{epstopdf}
\usepackage{color}
\usepackage{constants}
\usepackage{paralist}
\usepackage{mathrsfs}
\usepackage{comment}
\usepackage{amsrefs}

\makeatletter \@namedef{subjclassname@2000}{\emph{2020 Mathematics
Subject Classification}} \makeatother

\allowdisplaybreaks
\newtheorem{theorem}{Theorem}[section]

 \newtheorem{lemma}[theorem]{Lemma}
 \newtheorem{proposition}[theorem]{Proposition}

\numberwithin{equation}{section}

\newcommand{\dif}{\,\mathrm{d}}

\newcommand{\R}{\ensuremath{\mathbb{R}}}
\renewcommand{\S}{\ensuremath{\mathbb{S}}}
\newcommand{\N}{\ensuremath{\mathbb{N}}}
\newcommand{\Z}{\ensuremath{\mathbb{Z}}}

\begin{document}

\title{A new formula for the $L^p$ norm}

\pagestyle{plain}

\author[Q. Gu]{Qingsong Gu}
\address[Q. Gu]{Department of Mathematics, Nanjing University, Nanjing 210093, China} \email{\href{mailto: Qingsong Gu
<qingsonggu@nju.edu.cn>}{qingsonggu@nju.edu.cn} }

\author[P.-L. Yung]{Po-Lam Yung}
\address[P.-L. Yung]{Mathematical Sciences Institute, Australian National University, Canberra ACT 2601, Australia \quad {\it and} \quad Department of Mathematics, The Chinese University of Hong Kong, Shatin, Hong Kong}\email{\href{mailto: Po-Lam Yung
<PoLam.Yung@anu.edu.au>}{PoLam.Yung@anu.edu.au}} \email{\href{mailto: Po-Lam Yung
<plyung@math.cuhk.edu.hk>}{plyung@math.cuhk.edu.hk}}

\thanks {Yung is partially supported by a Future Fellowship FT200100399 from the Australian Research Council.}

\begin{abstract}
Recently, Brezis, Van Schaftingen and the second author \cite{BVY} established a new formula for the $\dot{W}^{1,p}$ norm of a function in $C^{\infty}_c(\R^N)$. The formula was obtained by replacing the $L^p(\R^{2N})$ norm in the Gagliardo semi-norm for $\dot{W}^{s,p}(\R^N)$ with a weak-$L^p(\R^{2N})$ quasi-norm and setting $s = 1$. This provides a characterization of such $\dot{W}^{1,p}$ norms, which complements the celebrated Bourgain-Brezis-Mironescu (BBM) formula \cite{BBM01}. In this paper, we obtain an analog for the case $s = 0$. In particular, we present a new formula for the $L^p$ norm of any function in $L^p(\R^N)$, which involves only the measures of suitable level sets, but no integration. This provides a characterization of the norm on $L^p(\R^N)$, which complements a formula by Maz\cprime ya and Shaposhnikova \cite{MSh}. As a result, by interpolation, we obtain a new embedding of the Triebel-Lizorkin space $F^s_{p,2}(\R^N)$ (i.e. the Bessel potential space $(I-\Delta)^{-s/2} L^p(\R^N)$), as well as its homogeneous counterpart $\dot{F}^s_{p,2}(\R^N)$, for $s \in (0,1)$, $p \in (1,\infty)$.
\end{abstract}

\dedicatory{In memory of Ka-Sing Lau}
\maketitle


\section{Introduction}

The purpose of this paper is to prove a new characterization of the $L^p$ norm on $\R^N$, by lifting to the product space $\R^N \times \R^N$ and considering a weak-$L^p$ quasi-norm over there instead. Indeed, for a measurable function $F(x,y)$ on $\R^N \times \R^N$ and $1 \leq p < \infty$, we denote the weak-$L^p$ quasi-norm of $F$ by $[F]_{L^{p,\infty}(\R^N \times \R^N)}$, where
\begin{equation}\label{Mar}
[F]_{L^{p,\infty}(\R^N \times \R^N)} := \sup_{\lambda > 0} \left( \lambda^p \mathcal{L}^{2N} \{(x,y) \in \R^{2N} \colon |F(x,y)| \geq \lambda\}\right)^{1/p}
\end{equation}
and $\mathcal{L}^{2N}$ denotes the Lebesgue measure on $\R^{2N}$ (see e.g., \cite{CR,G}). Then our first result reads:

\begin{theorem}\label{thm1.1}
For every $N \in \N$, there exist constants $c_1 = c_1(N) > 0$ and $c_2 = c_2(N) > 0$, such that for all $1 \leq p <\infty$ and all $u \in L^p(\R^N)$,
\begin{equation}\label{eq5'}
c_1^{1/p} \|u\|_{L^p(\R^N)}
\leq \left[\frac{u(x)-u(y)}{|x-y|^{\frac{N}{p}}}\right]_{L^{p,\infty}(\R^N \times \R^N)}
\leq 2 c_2^{1/p} \|u\|_{L^p(\R^N)}.
\end{equation}
Moreover, for $1 \leq p < \infty$ and $u \in L^p(\R^N)$, if we write
\begin{equation} \label{eq:E}
E_{\lambda} := \left\{ (x,y) \in \R^N \times \R^N \colon x \neq y, \, \frac{|u(x)-u(y)|}{|x-y|^{\frac{N}{p}}} \geq \lambda \right\},
\end{equation}
then
\begin{equation}\label{eq7'}
\lim_{\lambda \rightarrow 0^+}\lambda^p \mathcal{L}^{2N}(E_{\lambda})
= {2}\kappa_N\|u\|^p_{L^p(\R^N)},
\end{equation}
where $\kappa_N := \pi^{N/2}/\Gamma(\frac{N}{2}+1)$ is the volume of the unit ball in $\R^N$.
\end{theorem}

We remark that the power of $|x-y|$ in the denominator of the quantity in the middle of \eqref{eq5'} is the natural one dictated by dilation invariance. Furthermore, the main thrust of \eqref{eq5'} is in the first inequality. In fact, the second inequality has already been observed by e.g. Dominguez and Milman in \cite{DM}. On the other hand, the first inequality of \eqref{eq5'} is an easy consequence of \eqref{eq7'}, with $c_1(N) := 2 \kappa_N$ (because the supremum over $\lambda > 0$ always dominates the limit as $\lambda \rightarrow 0^+$). In addition, we emphasize that \eqref{eq7'} is not true unless we assume $u \in L^p(\R^N)$ to begin with; indeed, if $1 \leq p < \infty$ and $u$ is identically $1$, then $\mathcal{L}^{2N}(E_{\lambda}) = 0$ for every $\lambda > 0$, while $\|u\|_{L^p(\R^N)} = +\infty$. So the proof of \eqref{eq7'} is a little delicate, which we give in detail in Section~\ref{sec2}.

Our point of view of lifting to the product space $\R^N \times \R^N$ and using the weak-$L^p$ quasi-norm there is motivated by recent work of the second author with Ha\"{i}m Brezis and Jean Van Schaftingen \cite{BVY}, which established an analog of the above theorem for the Sobolev semi-norm $\|\nabla u\|_{L^p(\R^N)}$. The article \cite{BVY} in turn drew important inspiration from the BBM formula for the Sobolev space $W^{1,p}$, which first appeared in a celebrated paper \cite{BBM01} of Bourgain, Brezis and Mironescu. An analogue of the BBM formula for $L^p$ in place of $W^{1,p}$ was first obtained by Maz\cprime ya and Shaposhnikova \cite{MSh}. Our Theorem~\ref{thm1.1} can be thought of as a counterpart of the Maz\cprime ya-Shaposhnikova formula for the $L^p$ norm, in the same way that the main result in \cite{BVY} relates to the BBM formula for $W^{1,p}$.

To describe all these developments in more detail, let's introduce some notations. Let $\Omega$ be a  domain (i.e. an open, connected set) in $\R^N$. For $1 \leq p <\infty$ and $0 < s < 1$, the Gagliardo semi-norm of a function $u \in L^p(\Omega)$ is defined as
\begin{equation}\label{eq1}
|u|_{\dot{W}^{s,p}(\Omega)}:= \left( \int_{\Omega}\int_{\Omega} \frac{|u(x)-u(y)|^p}{|x-y|^{N+sp}}dxdy \right)^{1/p},
\end{equation}
where $|\cdot|$ in the denominator on the right hand side denotes the Euclidean norm on $\R^N$. (The dot above $W^{s,p}$ indicates that this semi-norm is homogeneous with respect to dilations.) This semi-norm is an important tool in the study of many partial differential equations, and has found numerous important applications (see e.g. \cite{BBM02,DPV,Maz}).

A well-known `defect' of this semi-norm is that $|u|_{\dot{W}^{s,p}(\Omega)}$ does not converge to the Sobolev semi-norm $\|\nabla u\|_{L^p(\Omega)}$ as $s\rightarrow 1^-$. Indeed, it is easy to see (c.f. \cite{BBM01}) that if $u$ is any smooth, non-constant function on a domain $\Omega \subset \R^N$, then $\|u\|^p_{\dot{W}^{s,p}(\Omega)} \rightarrow \infty$ as $s \rightarrow 1^-$ (see also \cite[Proposition 4]{Bre} for an extension to measurable $u$'s that are not necessarily smooth). This `defect' was addressed by Bourgain, Brezis and Mironescu in \cite{BBM01}: if $\Omega$ is a smooth, bounded domain in $\R^N$, then applying their Theorem~2 with
\begin{equation}%
\rho_s(x) = \frac{p(1-s) \mathbf{1}_{|x| \leq D}}{D^{p(1-s)} \omega_{N} |x|^{N-p(1-s)}}, \quad s \in (0,1)
\end{equation}%
where $D$ is the diameter of $\Omega$ and $\omega_{N}$ is the surface area of $\S^{N-1}$, we see that for $1 \leq p < \infty$ and $u\in W^{1,p}(\Omega) := \{u \in L^p(\Omega) \colon |\nabla u| \in L^p(\Omega)\}$, one has what is now known as the BBM formula:
\begin{equation}\label{eq2}
\lim\limits_{s\rightarrow1^-}(1-s) |u|^p_{\dot{W}^{s,p}(\Omega)}
= \frac{1}{p} k(p,N) \|\nabla u\|^p_{L^p(\Omega)},
\end{equation}
where
\begin{equation}\label{k}
k(p,N):=\int_{\S^{N-1}}|e \cdot \omega|^p d\omega
=\frac{2\Gamma((p+1)/2)\pi^{(N-1)/2}}{\Gamma((N+p)/2)}.
\end{equation}
Here $e\in \S^{N-1}$ is any fixed vector, $e\cdot \omega$ is the inner product of $e$ with $\omega$, and $d\omega$ is the surface measure on $\S^{N-1}$ induced from the Lebesgue measure on $\R^N$. See also D\'{a}vila \cite{Dav} for an extension to the space of functions of bounded variation on $\Omega$.

On the other hand, for $s \in (0,1)$ and $1 \leq p < \infty$, let $W^{s,p}_0(\R^N)$ be the completion of $C^{\infty}_c(\R^N)$ under the Gagliardo semi-norm $|\cdot|_{\dot{W}^{s,p}(\R^N)}$. Parallel to the BBM formula \eqref{eq2}, Maz\cprime ya and Shaposhnikova \cite{MSh} showed that for any $u\in \bigcup_{0<s<1} W_0^{s,p}(\R^N)$, we have
\begin{equation}\label{MSh}
\lim\limits_{s\rightarrow0^+}s\|u\|^p_{\dot{W}^{s,p}(\R^N)}
=\frac{2N}{p} \kappa_N \|u\|^p_{L^p(\R^N)},
\end{equation}
where $\kappa_N$ is the volume of the unit ball in $\R^N$.

Recently, Brezis, Van Schaftingen and the second author \cite{BVY} considered what happened when one replaces the $L^p$ norm on $\R^N \times \R^N$ in the Gagliardo semi-norm $|\cdot|_{\dot{W}^{s,p}(\R^N)}$ by the weak-$L^p$ quasi-norm, and evaluates it at $s = 1$. This leads to the following characterization of $\|\nabla u\|_{L^p(\R^N)}$ in \cite[Theorem 1.1]{BVY}: they proved the existence of two positive constants $c=c(N)$ and $C=C(N)$ such that for all $u \in C_c^\infty(\R^N)$ and $1 \leq p < \infty$,
\begin{equation}\label{eq5}
c^p \|\nabla u\|^p_{L^p(\R^N)}
\leq \left[\frac{u(x)-u(y)}{|x-y|^{\frac{N}{p}+1}}\right]^p_{L^{p,\infty}(\R^N \times \R^N)}
\leq C\|\nabla u\|^p_{L^p(\R^N)}.
\end{equation}
Furthermore, it was shown that for $u \in C^{\infty}_c(\R^N)$ and $1 \leq p < \infty$, if
\begin{equation}%
\tilde{E}_{\lambda} := \left\{(x,y) \in \R^N \times \R^N \colon x \neq y, \, \frac{|u(x)-u(y)|}{|x-y|^{\frac{N}{p}+1}}\geq\lambda \right\},
\end{equation}%
then
\begin{equation}\label{eq7}
\lim_{\lambda\rightarrow\infty}\lambda^p \mathcal{L}^{2N}(\tilde{E}_{\lambda})
=\frac{1}{N} k(p,N) \|\nabla u\|^p_{L^p(\R^N)}.
\end{equation}
Thus, the first inequality in \eqref{eq5}, with $c(N):=\inf_{p \in [1,\infty)} (k(p,N)/N)^{1/p} > 0$, is a direct consequence of \eqref{eq7} and \eqref{Mar}. See also Poliakovsky \cite[Lemma 3.1]{Po} for an extension of the second inequality in \eqref{eq5} to functions $u \in W^{1,p}(\R^N) := \{u \in L^p(\R^N) \colon |\nabla u| \in L^p(\R^N)\}$.

In light of the formula \eqref{MSh} of Maz\cprime ya and Shaposhnikova mentioned above, which establishes an analog of the BBM formula \eqref{eq2} when $s \rightarrow 0^+$, a natural question is whether one has an analog of \eqref{eq5} and \eqref{eq7} for $L^p$ instead of $W^{1,p}$. Our Theorem~\ref{thm1.1} can be thought of as an affirmative answer to this question. Our proof is technically simpler than the corresponding one for \eqref{eq5} and \eqref{eq7} in \cite{BVY}, in that our proof relies only on Fubini's theorem, but not on any covering lemma nor any Taylor expansion. On the other hand, it came as a mild surprise that while \eqref{eq7} involves a limit as $\lambda \rightarrow +\infty$, its cousin \eqref{eq7'} involves instead a limit where $\lambda \rightarrow 0^+$: the former is natural since large values of $\lambda$ captures what happens to $|u(x)-u(y)|$ when $x$ and $y$ are close to each other, which in turn relates to the size of $|\nabla u(x)|$, but we do not have a good explanation of the latter.

We next turn to two results obtained by interpolating the upper bound in \eqref{eq5'}, with the upper bound in \eqref{eq5}. The first result can be formulated using the Bessel potential spaces $(I-\Delta)^{-s/2} L^p(\R^N)$:

\begin{theorem} \label{thm:1.2'}
For every $N \in \N$ and $p \in (1,\infty)$, there exists a constant $C' = C'(p,N)$ such that for all $s \in (0,1)$ and all $u \in (I-\Delta)^{-s/2} L^p(\R^N) $, we have
\begin{equation}%
\left[\frac{u(x)-u(y)}{|x-y|^{\frac{N}{p}+s}} \right]_{L^{p,\infty}(\R^N \times \R^N)} \leq C'  \|(I-\Delta)^{s/2} u\|_{L^p(\R^N)}.
\end{equation}%
\end{theorem}

This theorem follows from complex interpolation by considering the following holomorphic family of linear operators
\begin{equation} \label{eq:Tzdef}
u(x) \mapsto T_z u(x,y) := \frac{u(x)-u(y)}{|x-y|^{\frac{N}{p}+z}}
\end{equation}
where $z \in \mathbb{C}$ takes value in the strip $\{0 \leq \text{Re}\, z \leq 1\}$. Indeed, the second inequality in \eqref{eq5'} shows that when $\textrm{Re}\, z = 0$, $T_z$ maps $L^p(\R^N)$ to $L^{p,\infty}(\R^N \times \R^N)$. On the other hand, as observed by Poliakovsky \cite[Lemma 3.1]{Po}, the second inequality in \eqref{eq5} continues to hold for all $u \in W^{1,p}(\R^N) = (I-\Delta)^{-1/2} L^p(\R^N)$.  Thus when $\textrm{Re}\, z = 1$, $T_z$ maps $(I-\Delta)^{-1/2} L^p(\R^N)$ to $L^{p,\infty}(\R^N \times \R^N)$. Theorem~\ref{thm:1.2'} now follows by complex interpolation.

One drawback of Theorem~\ref{thm:1.2'} is that the left-hand side concerns a homogeneous norm, while the right-hand side contains an inhomogeneous norm. But it is only slightly harder to prove a variant of Theorem~\ref{thm:1.2'}, concerning a homogeneous Triebel-Lizorkin space instead.

First, let's recast Theorem~\ref{thm:1.2'} in terms of (inhomogeneous) fractional Triebel-Lizorkin spaces $F^{s}_{p,q}$ on $\R^N$, which we define as follows. Let $\mathcal{S}(\R^N)$ be the Fr\'{e}chet space of Schwartz functions on $\R^N$, and $\mathcal{S}'(\R^N)$ the space of all tempered distributions on $\R^N$. Let $\mathcal{F}^{-1}$ be the inverse Fourier transform on $\R^N$ given by
\begin{equation}%
\mathcal{F}^{-1} \phi(x) = \int_{\R^N} \phi(\xi) e^{2\pi i x \cdot \xi} d\xi,
\end{equation}%
for $\phi \in \mathcal{S}(\R^N)$. Let $\varphi \in C^{\infty}_c(\R^N)$ be a fixed function supported on $\{|\xi| \leq 2\}$ such that $\varphi(\xi) = 1$ whenever $|\xi| \leq 1$. Write
\begin{equation} \label{psidef}
\psi(\xi) = \varphi(\xi) - \varphi(2\xi)
\end{equation}
so that $\psi \in C^{\infty}_c(\R^N)$ is supported on $\{1/2 \leq |\xi| \leq 2\}$ with
\begin{equation}%
\varphi(\xi) + \sum_{j \in \N} \psi(2^{-j} \xi) = 1 \quad \text{for all $\xi \in \R^N$}.
\end{equation}%
A corresponding family of Littlewood-Paley projections is given by
\begin{equation}%
P_0 u(x) := u * \mathcal{F}^{-1} \varphi (x)
\end{equation}%
and
\begin{equation}%
\Delta_j u(x) := u * \mathcal{F}^{-1} \psi_j (x),
\end{equation}%
where $\psi_j(\xi) := \psi(2^{-j} \xi)$.
For $s \in \R$, $p \in (1,\infty)$ and $q \in (1,\infty)$, we define the \emph{(inhomogeneous) Triebel-Lizorkin space} $F^{s}_{p,q}(\R^N)$ to be the space of all $u \in \mathcal{S}'(\R^N)$ for which
\begin{equation}%
\|u\|_{F^{s}_{p,q}(\R^N)} := \Big\| \Big( |P_0 u|^q + \sum_{j \in \N} |2^{js} \Delta_j u|^q \Big)^{1/q} \Big\|_{L^p(\R^N)} < \infty.
\end{equation}%
Standard Littlewood-Paley theory shows that
for $1 < p < \infty$, $s \in \R$, we have
\begin{equation}%
F^{s}_{p,2}(\R^N) = (I-\Delta)^{-s/2} L^p(\R^N)
\end{equation}%
with comparable norms: for all $u \in \mathcal{S}'(\R^N)$, we have
\begin{equation}%
\|u\|_{F^{s}_{p,2}(\R^N)} \simeq_{p,N} \|(I-\Delta)^{s/2} u\|_{L^p(\R^N)}.
\end{equation}%
Thus we could have replaced the Bessel potential spaces $(I-\Delta)^{-s/2} L^p(\R^N)$ in Theorem~\ref{thm:1.2'} by the inhomogeneous $F^{s}_{p,2}(\R^N)$.

This motivates us to consider a variant of Theorem~\ref{thm:1.2'} for homogeneous Triebel-Lizorkin spaces instead. To introduce these spaces, we denote by $\mathcal{Z}(\R^N)$ the subspace of all $u \in \mathcal{S}(\R^N)$ for which $\int_{\R^N} u(x) p(x) dx = 0$ for every polynomial $p(x) \in \R[x]$, and denote by $\mathcal{Z}'(\R^N)$ the space of all continuous linear functionals on $\mathcal{Z}(\R^N)$, which we identify with the quotient $\mathcal{S}'(\R^N) / \{\text{polynomials on $\R^N$}\}$. If $\psi \in \mathcal{S}(\R^N)$ is as in \eqref{psidef} and $\psi_j(\xi) := \psi(2^{-j} \xi)$ for $j \in \Z$, then
\begin{equation}%
\sum_{j \in \Z} \psi_j(\xi) = 1 \quad \text{for all $\xi \in \R^N \setminus \{0\}$}.
\end{equation}%
We denote by $\{\Delta_j\}_{j \in \Z}$ the family of Littlewood-Paley projections given by
\begin{equation}%
\Delta_j u(x) := u * \mathcal{F}^{-1} \psi_j (x),
\end{equation}%
which is well-defined for all $u \in \mathcal{Z}'(\R^N)$  (because $\int_{\R^N} \mathcal{F}^{-1} \psi_j(x) p(x) dx = 0$ for all polynomials $p(x) \in \R[x]$.) The \emph{homogeneous Triebel-Lizorkin space} $\dot{F}^{s}_{p,q}(\R^N)$ is then defined to be the space of all $u \in \mathcal{Z}'(\R^N)$ for which
\begin{equation}%
\|u\|_{\dot{F}^{s}_{p,q}(\R^N)} := \Big\| \Big( \sum_{j \in \Z} |2^{js} \Delta_j u|^q \Big)^{1/q} \Big\|_{L^p(\R^N)} < \infty.
\end{equation}%

It was known (c.f. proof of Theorem in \cite[Chapter 5.1.5]{Triebel}) that $\mathcal{F}^{-1} [C^{\infty}_c(\R^N \setminus \{0\})]$, the space of (Schwartz) functions on $\R^N$ given by inverse Fourier transforms of $C^{\infty}$, compactly supported functions on $\R^N \setminus \{0\}$, is a dense subset of $\dot{F}^{s}_{p,q}(\R^N)$ for $s \in \R$, $p \in (1,\infty)$ and $q \in (1,\infty)$ (see Appendix below for a sketch of proof). Also, for $1 < p < \infty$, we have
\begin{equation}%
\|u\|_{\dot{F}^{s}_{p,2}(\R^N)} \simeq_{p,N}
\begin{cases}
\|u\|_{L^p(\R^N)} & \quad \text{if $s = 0$}\\
\|\nabla u\|_{L^p(\R^N)} & \quad \text{if $s = 1$},
\end{cases}
\end{equation}%
at least if $u \in \mathcal{F}^{-1} [C^{\infty}_c(\R^N \setminus \{0\})]$ (indeed this holds as long as $u \in \mathcal{S}'(\R^N)$ for which the right hand side of the above display equation is finite). This allows us to prove the next result, concerning the homogeneous space $\dot{F}^{s}_{p,2}(\R^N)$:
\begin{theorem}\label{th1.2}
For every $N \in \N$ and $p \in (1,\infty)$, there exists a constant $C'=C'(p,N)$ such that for all $s \in (0,1)$ and all $u \in \mathcal{F}^{-1} [C^{\infty}_c(\R^N \setminus \{0\})]$,
\begin{equation}\label{eq:new}
\left[\frac{u(x)-u(y)}{|x-y|^{\frac{N}{p}+s}} \right]_{L^{p,\infty}(\R^N \times \R^N)} \leq C'  \|u\|_{\dot{F}^{s}_{p,2}(\R^N)}.
\end{equation}
As a result, for $s \in (0,1)$ and $p \in (1,\infty)$, one may define the left-hand side of \eqref{eq:new} for all $u \in \dot{F}^{s}_{p,2}(\R^N)$ by density, and the inequality \eqref{eq:new} continues to hold.
\end{theorem}

We note that Dominguez and Milman \cite[Theorem 4.1]{DM} had actually proved a stronger embedding, namely that the $\dot{F}^{s}_{p,2}$ norm above can be replaced by $\dot{F}^{s}_{p,\infty}$, but their constant might blow up as $s \to 0^+$, whereas ours remain bounded uniformly for all $0 < s < 1$.

Theorem~\ref{th1.2} is the most powerful in the case $1 < p < 2$, as one can see by comparing \eqref{eq:new} with the following known inequality for $\dot{F}^{s}_{p,p}(\R^N)$ (so $q = p$ as opposed to $q = 2$ in Theorem~\ref{th1.2}):
\begin{proposition}\label{th1.3}
For every $N \in \N$ and every $p \in (1,\infty)$, there exists a constant $C'' = C''(N,p)$ so that for all $u \in \mathcal{F}^{-1} [C^{\infty}_c(\R^N \setminus \{0\})]$, $s \in (0,1)$, one has
\begin{equation} \label{eq:known1}
\left\| \frac{u(x)-u(y)}{|x-y|^{\frac{N}{p}+s}} \right\|_{L^p(\R^N \times \R^N)} \leq C'' \Big[\frac{1}{s(1-s)}\Big]^{\max\{\frac{1}{2},\frac{1}{p}\}} \|u\|_{\dot{F}^{s}_{p,p}(\R^N)}.
\end{equation}
\end{proposition}
The left hand side of \eqref{eq:new} is smaller than the left hand side of \eqref{eq:known1} by Chebyshev's inequality, but the norm on the right hand side of \eqref{eq:new} is also smaller than the norm on the right hand side of \eqref{eq:known1} if $1 < p < 2$ (because $\|u\|_{\dot{F}^{s}_{p,2}} \leq \|u\|_{\dot{F}^{s}_{p,p}}$ if $p < 2$). In addition, the constant $C'$ in \eqref{eq:new} does not blow up if we fix $p$ and let $s \rightarrow 0^+$ or $1^-$.

The proof of Theorem \ref{th1.2} will be given in Section~\ref{sec3}. For the convenience of the reader, we will also give a proof of Proposition \ref{th1.3}, which we adapt from \cite[Chapter V.5]{S70}. We also remark in passing that it is also known that if $p \in [2,\infty)$, then the right hand side of \eqref{eq:known1} can also be replaced by $C''(N,p) [s(1-s)]^{-1/p} \|u\|_{\dot{F}^{s}_{p,2}(\R^N)}$. In other words, if we control $\|u\|_{\dot{F}^{s}_{p,2}(\R^N)}$ (which is typically bigger than the norm $\|u\|_{\dot{F}^{s}_{p,p}(\R^N)}$ that appears in \eqref{eq:known1}), then to control the left hand side of \eqref{eq:known1}, we only need to pay a price of a smaller constant $C''(N,p) [s(1-s)]^{-1/p}$ (as opposed to $C''(N,p) [s(1-s)]^{-1/2}$). This follows from an adaptation of the arguments given for Proposition~\ref{th1.3}, which we will not give in detail.

An interesting related question is whether the inequality in \eqref{eq:new} can be reversed. Since $\dot{F}^{s}_{p,2}(\R^N) = [L^p(\R^N), \dot{W}^{1,p}(\R^N)]_s$, this question could be reformulated as follows: If $N \in \N$, $p \in (1, \infty)$, $s \in (0,1)$ and $u \in \mathcal{F}^{-1} [C^{\infty}_c(\R^N \setminus \{0\})]$, is there a holomorphic family of functions $\{u_z(x) \colon z \in \mathbb{C}, 0 \leq \text{Re}\, z \leq 1 \}$ so that $u_s(x) = u(x)$, and so that
\begin{equation}%
\max \left\{ \sup_{\text{Re}\, z = 0} \|u_z\|_{L^p(\R^N)}, \sup_{\text{Re}\, z = 1} \|\nabla u_z\|_{L^p(\R^N)} \right\} \lesssim \left[ \frac{u(x)-u(y)}{|x-y|^{\frac{N}{p}+s}} \right]_{L^{p,\infty}(\R^N \times \R^N)}?
\end{equation}%

\noindent{\textbf{Acknowledgements.}} The authors thank Ha\"{i}m Brezis and Jean Van Schaftingen for their kind encouragement as we pursued this project, and Armin Schikorra for sharing his thoughts related to Proposition~\ref{th1.3}. They also thank Ka-Sing Lau for his teaching and inspiration over the years.

\section{Proof of Theorem \ref{thm1.1}}\label{sec2}
\begin{proof}
As remarked above, the second inequality in \eqref{eq5'} is essentially known. It was stated without proof in \cite{DM}. But for completeness, and also because we need to use it to derive the first inequality in \eqref{eq5'}, we give its simple proof below. Indeed, we show that for $1 \leq p < \infty$ and all measurable functions $u$ on $\R^N$,
\begin{equation}\label{eqeq5'}
\left[\frac{u(x)-u(y)}{|x-y|^{\frac{N}{p}}}\right]^p_{L^{p,\infty}(\R^N\times \R^N)}\leq 2^{p+1}\kappa_N\|u\|^p_{L^p(\mathbb R^N)}
\end{equation}
so that the second inequality of \eqref{eq5'} holds with $c_2(N) := 2 \kappa_N$, where $\kappa_N$ is the volume of the unit ball in $\R^N$.

To prove \eqref{eqeq5'}, given $1 \leq p < \infty$, a measurable $u$ on $\R^N$, and $\lambda > 0$, let $E_{\lambda}$ be as in \eqref{eq:E}. Then by the triangle inequality,
\begin{equation}%
\begin{split}
E_{\lambda} \subset & \left\{(x,y)\in \mathbb R^N\times\mathbb R^N \colon x\neq y, \, |u(x)|\geq\frac12\lambda|x-y|^{N/p}\right\} \\
& \quad \quad \bigcup \left\{(x,y)\in \mathbb R^N\times\mathbb R^N \colon x\neq y, \, |u(y)|\geq\frac12\lambda|x-y|^{N/p}\right\}
\end{split}
\end{equation}%
so
\begin{equation}%
\begin{split}
\mathcal L^{2N}(E_\lambda)
&\leq \int_{\mathbb R^N}\int_{\mathbb R^N} \mathbf{1}_{\left\{(x,y) \colon |y-x| \leq (2|u(x)|\lambda^{-1})^{p/N}\right\}}dydx \\
&\quad \quad +\int_{\mathbb R}\int_{\mathbb R^N} \mathbf{1}_{\left\{(x,y) \colon |y-x|\leq (2|u(y)|\lambda^{-1})^{p/N}\right\}}dxdy \\
&=\kappa_N\int_{\mathbb R^N}(2\lambda^{-1})^{p}|u(x)|^pdx+\kappa_N\int_{\mathbb R^N}(2\lambda^{-1})^{p}|u(y)|^pdy \\
&=2^{p+1}\kappa_N\lambda^{-p}\|u\|_{L^p(\mathbb R^N)}^p.
\end{split}
\end{equation}%
\eqref{eqeq5'} now follows by multiplying by $\lambda^p$ on both sides and taking supremum over all $\lambda > 0$.

It remains to establish \eqref{eq7'} for all $u \in L^p(\R^N)$, $1 \leq p < \infty$, which would then imply the first inequality in \eqref{eq5'}. We first consider the case under the additional assumption that $u$ is compactly supported on $\R^N$. This extra assumption about $u$ will then be removed by using suitable truncations of $u$, together with \eqref{eqeq5'} which handles the error that arises.

\noindent\textbf{Case 1. $u$ is compactly supported.}
For $\lambda > 0$, let $E_{\lambda}$ be as in \eqref{eq:E}. Then
\begin{equation}\label{eq''}
\mathcal L^{2N}(E_{\lambda})=2\mathcal L^{2N}(H_{\lambda})
\end{equation}
where
\begin{equation}
H_{\lambda} := E_{\lambda} \cap \{(x,y) \in \R^N \times \R^N \colon |y|>|x|\}.
\end{equation}
This is because $E_{\lambda}$ is the union of its three subsets, one where $|y| > |x|$, one where $|y| < |x|$, and one where $|y| = |x|$. The last set has $\mathcal{L}^{2N}$ measure zero, and the first two sets have the same $\mathcal{L}^{2N}$ measure by symmetry of the set $E_{\lambda}$. Hence we only need to estimate $\mathcal{L}^{2N}(H_\lambda)$.

Since $u$ is compactly supported, we may assume
\begin{equation}%
\text{supp} \, u \subseteq B_R := \{x \in \R^N \colon |x| < R\}
\end{equation}%
for some $R > 0$. Now if $(x,y) \in H_{\lambda}$, then we must have $x \in B_R$. This is because otherwise both $x,y$ are outside $B_R$, which by our assumption about the support of $u$ implies that $u(x)=u(y)=0$, and hence $(x,y) \notin E_{\lambda}$, contradicting that $(x,y) \in H_{\lambda}$. Moreover, for $x \in B_R$, let
\begin{equation}%
H_{\lambda,x} :=\left\{y \in \R^N \colon |y|>|x|, \,\frac{|u(y)-u(x)|}{|y-x|^{N/p}}\geq\lambda\right\}
\end{equation}%
and
\begin{equation}%
H_{\lambda,x,R} := \left\{y \in \R^N \colon |y| \geq R, \, |y-x|\leq \left(\frac{|u(x)|}{\lambda}\right)^{p/N}\right\}.
\end{equation}%
Then Fubini's theorem gives
\begin{equation} \label{eq:Fub}
\mathcal{L}^{2N}(H_{\lambda}) = \int_{B_R} \mathcal{L}^N(H_{\lambda,x}) dx,
\end{equation}
while
\begin{equation}%
H_{\lambda,x,R} = H_{\lambda,x} \setminus B_R,
\end{equation}%
because for $|y| \geq R$, we have $u(y) = 0$ and hence $|u(x)| = |u(y)-u(x)|$. It follows that
\begin{equation} \label{eq:Hinclusion}
H_{\lambda,x,R} \subseteq H_{\lambda,x} \subseteq H_{\lambda,x,R} \cup B_R.
\end{equation}
Writing $\kappa_N = \mathcal{L}^N(B_1)$, from the first inclusion in \eqref{eq:Hinclusion}, we have
\begin{equation}\label{eqeqeq3}
\mathcal{L}^N(H_{\lambda,x})
\geq \mathcal{L}^N(H_{\lambda,x,R})
\geq \kappa_N \frac{|u(x)|^p}{\lambda^p} - \kappa_N R^N.
\end{equation}
On the other hand, from the second inclusion in \eqref{eq:Hinclusion}, we have
\begin{equation}\label{eqeqeq3'}
\mathcal{L}^N(H_{\lambda,x})
\leq \kappa_N \frac{|u(x)|^p}{\lambda^p} + \kappa_N R^N.
\end{equation}
Integrating \eqref{eqeqeq3} and \eqref{eqeqeq3'} over $x \in B_R$, and using \eqref{eq:Fub}, we obtain
\begin{equation}%
\frac{\kappa_N}{\lambda^p} \|u\|^p_{L^p(\R^N)} - \kappa_N^2 R^{2N} \leq \mathcal{L}^{2N}(H_{\lambda}) \leq \frac{\kappa_N}{\lambda^p} \|u\|^p_{L^p(\R^N)} + \kappa_N^2 R^{2N}.
\end{equation}%
Multiplying both sides by $\lambda^p$ and letting $\lambda \rightarrow 0^+$, we have
\begin{equation}%
\lim_{\lambda \rightarrow 0^+} \lambda^p \mathcal{L}^{2N}(H_{\lambda}) = \kappa_N \|u\|^p_{L^p(\R^N)},
\end{equation}%
as desired.

\bigskip

\noindent\textbf{Case 2. $u \in L^p(\R^N)$, not necessarily compactly supported.} Let $u_R = u \cdot \mathbf{1}_{B_R}$ be the truncation of $u$ with $|x|\leq R$ for some $R>0$. Let $v_R=u-u_R$. Later we will crucially use that $\|v_R\|_{L^p(\mathbb R^N)}\rightarrow0$ as $R \rightarrow \infty$, which holds only because $u \in L^p(\R^N)$ and $1 \leq p < \infty$.

Now since $u = u_R + v_R$, for any $\sigma\in(0,1)$, we have
\begin{equation} \label{eq:Edec1}
E_{\lambda} = \left\{(x,y) \in \R^N \times \R^N \colon \frac{|u(x)-u(y)|}{|x-y|^{N/p}}\geq\lambda\right\} \subseteq A_1 \cup A_2
\end{equation}
where
\begin{equation}
A_1 := \left\{(x,y) \in \R^N \times \R^N \colon \frac{|u_R(x)-u_R(y)|}{|x-y|^{N/p}}\geq\lambda(1-\sigma)\right\}
\end{equation}
and
\begin{equation} \label{eq:A2}
A_2 := \left\{(x,y) \in \R^N \times \R^N \colon \frac{|v_R(x)-v_R(y)|}{|x-y|^{N/p}}\geq {\lambda\sigma}\right\}.
\end{equation}
Hence
\begin{equation}\label{eqeqeq17'}
\mathcal L^{2N}(E_\lambda)\leq \mathcal L^{2N}(A_1)+\mathcal L^{2N}(A_2).
\end{equation}
Since $u_R$ is compactly supported in $B_R$, by \eqref{eqeqeq3'} with $\lambda$ replaced by $\lambda(1-\sigma)$, we obtain
\begin{equation}\label{eqeqeq11}
\mathcal L^{2N}(A_1)\leq \frac{2\kappa_N}{\lambda^p(1-\sigma)^p}\|u_R\|_{L^p(\mathbb R^N)}^p+2(\kappa_NR^N)^2.
\end{equation}
For $A_2$, by using \eqref{eqeq5'} for $v_R$, we obtain
\begin{equation}
\mathcal L^{2N}(A_2)\leq\frac{2^{p+1}\kappa_N}{(\lambda\sigma)^{p}}\|v_R\|^p_{L^p(\mathbb R^N)}.\label{eqeqeq9}
\end{equation}
Combining \eqref{eqeqeq17'}, \eqref{eqeqeq11} and \eqref{eqeqeq9}, and multiplying by $\lambda^p$, we obtain
\begin{equation}\label{eqeqeq15}
\lambda^p\mathcal L^{2N}(E_{\lambda})\leq \frac{2\kappa_N}{(1-\sigma)^p}\|u_R\|_{L^p(\mathbb R^N)}^p+2\lambda^p(\kappa_NR^N)^2+\frac{2^{p+1}\kappa_N}{\sigma^p}\|v_R\|_{L^p(\mathbb R^N)}^p.
\end{equation}
We now first let $\lambda\rightarrow 0^+$, then let $R\rightarrow\infty$ and finally let $\sigma\rightarrow 0^+$. Since
\begin{equation}%
\lim_{R \rightarrow \infty} \|u_R\|_{L^p(\R^N)} = \|u\|_{L^p(\R^N)} \quad \text{and} \quad \lim_{R \rightarrow \infty} \|v_R\|_{L^p(\R^N)} = 0,
\end{equation}%
we obtain
\begin{equation}\label{eqeqeq16}
\limsup_{\lambda\rightarrow0^+}\lambda^p\mathcal L^{2N}(E_{\lambda})\leq 2\kappa_N\|u\|_{L^p(\mathbb R^N)}^p.
\end{equation}

Similarly, for any $\sigma > 0$, we have
\begin{equation}%
E_{\lambda}=\left\{(x,y) \in \R^N \times \R^N \colon \frac{|u(x)-u(y)|}{|x-y|^{N/p}}\geq\lambda\right\}\supseteq A_3 \setminus A_2
\end{equation}%
where
\begin{equation}
A_3 := \left\{(x,y) \in \R^N \times \R^N \colon \frac{|u_R(x)-u_R(y)|}{|x-y|^{N/p}}\geq\lambda(1+\sigma)\right\}
\end{equation}
and $A_2$ is as in \eqref{eq:A2}. Hence
\begin{equation}\label{eqeqeq17}
\mathcal L^{2N}(E_\lambda)\geq \mathcal L^{2N}(A_3)-\mathcal L^{2N}(A_2).
\end{equation}
Since $u_R$ is compactly supported in $B_R$, by \eqref{eqeqeq3} with $\lambda$ replaced by $\lambda(1+\sigma)$, we have
\begin{equation}\label{eqeqeq18}
\mathcal L^{2N}(A_3)\geq \frac{2\kappa_N}{\lambda^p(1+\sigma)^p}\|u_R\|_{L^p(\mathbb R^N)}^p-2(\kappa_NR^N)^2.
\end{equation}
Combining \eqref{eqeqeq17}, \eqref{eqeqeq18} and \eqref{eqeqeq9}, and multiplying by $\lambda^p$, we obtain
\begin{equation}\label{eqeqeq19}
\lambda^p\mathcal L^{2N}(E_{\lambda})\geq \frac{2\kappa_N}{(1+\sigma)^p}\|u_R\|_{L^p(\mathbb R^N)}^p-2\lambda^p(\kappa_NR^N)^2-\frac{2^{p+1}\kappa_N}{\sigma^p}\|v_R\|_{L^p(\mathbb R^N)}^p.
\end{equation}
We now first let $\lambda\rightarrow 0^+$, then let $R \rightarrow\infty$ and finally let $\sigma \rightarrow 0^+$. We obtain
\begin{equation}\label{eqeqeq20}
\liminf_{\lambda\rightarrow0^+}\lambda^p\mathcal L^{2N}(E_{\lambda})\geq 2\kappa_N\|u\|_{L^p(\mathbb R^N)}^p.
\end{equation}
\eqref{eq7'} then follows from \eqref{eqeqeq16} and \eqref{eqeqeq20}.
\end{proof}

\section{Embeddings of homogeneous fractional Triebel-Lizorkin spaces}\label{sec3}

In this section, we first prove Theorem~\ref{th1.2}. Its proof is similar to that of Theorem~\ref{thm:1.2'}, in that it also proceeds via complex interpolation, for the holomorphic family of linear operators $\{T_z\}$ defined in \eqref{eq:Tzdef}. On the other hand, it is not clear whether $T_1$ maps $\dot{F}^{1}_{p,2}(\R^N)$ to $L^{p,\infty}(\R^N \times \R^N)$. Thus we provide a more careful proof below, explaining why interpolation works.

First, we recall the subspace $\mathcal{F}^{-1} [C^{\infty}_c(\R^N \setminus \{0\})]$ which consists of Schwartz functions on $\R^N$ and is dense in $\dot{F}^{s}_{p,2}(\R^N)$ when $s \in (0,1)$, $p \in (1,\infty)$. For $u \in \mathcal{F}^{-1} [C^{\infty}_c(\R^N \setminus \{0\})]$, say $u = \mathcal{F}^{-1} \widehat{u}$ and $\widehat{u} \in C^{\infty}_c(\R^N \setminus \{0\})$, we may define complex powers of Laplacian:
\begin{equation} \label{eq:Dpow}
(-\Delta)^z u(x) := \int_{\R^N} (2 \pi |\xi|)^{2z} \widehat{u}(\xi) e^{2 \pi i x \cdot \xi} d\xi, \quad z \in \mathbb{C}.
\end{equation}
For every fixed $x \in \R^N$, this defines an entire function of $z \in \mathbb{C}$. Furthermore, for every fixed $z \in \mathbb{C}$, this defines a function in $\mathcal{F}^{-1} [C^{\infty}_c(\R^N \setminus \{0\})] \subset \mathcal{S}(\R^N)$.

\begin{lemma}\label{thm3.2}
For every $N \in \N$ and $1 < p < \infty$, there exists a constant $A=A(p,N)$ such that for any $u \in \mathcal{F}^{-1} [C^{\infty}_c(\R^N \setminus \{0\})]$ and any $s \in (0,1)$, the following estimates hold.
\begin{enumerate}[(a)]
\item For $z \in \mathbb{C}$ with $\text{Re}\, z = 0$, we have
\begin{equation}\label{eq:Rez=0}
\|(-\Delta)^{(s-z)/2} u\|_{L^p(\R^N)}\leq A(1+|\text{Im}\, z|)^{N+1} \|u\|_{\dot{F}^{s}_{p,2}(\R^N)}.
\end{equation}
\item For $z \in \mathbb{C}$ with $\text{Re}\, z = 1$, we have
\begin{equation}\label{eq:Rez=1}
\|\nabla(-\Delta)^{(s-z)/2} u\|_{L^p(\R^N)}\leq A(1+|\text{Im}\, z|)^{N+1} \|u\|_{\dot{F}^{s}_{p,2}(\R^N)}.
\end{equation}
\end{enumerate}
\end{lemma}

\begin{proof}
The proof is a standard application of the theory of singular integrals. We verify the $z$-dependence of the constants in \eqref{eq:Rez=0} and \eqref{eq:Rez=1} by providing the necessary details below.

We first prove (a). Let $z \in \mathbb{C}$ with $\text{Re}\, z = 0$. Write $z = it$ for some $t \in \R$. Let $\{\tilde{\Delta}_j\}_{j \in \Z}$ be another family of Littlewood-Paley projections, given by
\begin{equation}%
\tilde{\Delta}_j u(x) := u * \mathcal{F}^{-1} \tilde{\psi}_j(x),
\end{equation}%
where $\tilde{\psi}_j(\xi) := \tilde{\psi}(2^{-j} \xi)$ for some $C^{\infty}$ function supported on $\{1/4 \leq |\xi| \leq 4\}$, so that $\tilde{\psi}(\xi) = 1$ on the support of $\psi$; this gives $\tilde{\Delta}_j \Delta_j = \Delta_j$ for all $j \in \Z$. As a result, for $u \in \mathcal{F}^{-1} [C^{\infty}_c(\R^N \setminus \{0\})]$, $j \in \Z$ and $s \in (0,1)$, we have
\begin{equation} \label{eq:reproduce}
\Delta_j (-\Delta)^{(s-z)/2} u = (2^{js} \Delta_j u)*K_j
\end{equation}
where $K_j := \mathcal{F}^{-1} [2^{-js} (2 \pi |\xi|)^{s- it} \tilde{\psi}_j]$ satisfies
\begin{equation}\label{eq3.5}
\sup_{j \in \Z} |\nabla K_j(x)| \lesssim_N (1+|t|)^{N+1} |x|^{-(N+1)}.
\end{equation}
(This is because for any multiindices $\alpha$, one has
\begin{equation}\label{eq3.4}
|\partial_{\xi}^{\alpha} [2^{-js} {(2 \pi |\xi|)}^{s- it} \tilde{\psi}_j(\xi)]| \lesssim_{\alpha} (1+|t|)^{|\alpha|} 2^{-j|\alpha|} \chi_{|\xi| \simeq 2^j}
\end{equation}
with implicit constant independent of $j \in \Z$, $s \in (0,1)$ and $t \in \R$; we may apply this with $|\alpha| = N$ and $N+1$ to bound $|x|^{N+1} |\nabla K_j(x)|$ in $L^{\infty}(\R^N)$.)
We may now apply a vector-valued singular integral theorem to the operator
\begin{equation} \label{eq:vs}
(f_j(x))_{j \in \Z} \mapsto (f_j*K_j(x))_{j \in \Z},
\end{equation}
which is clearly bounded on $L^2(\ell^2)$ with norm $\lesssim 1$; by \cite[Chapter II, Theorem 5]{S70}, or \cite[Chapter I.6.4]{S93}, this operator is also bounded on $L^p(\ell^2)$ for all $1 < p < \infty$, with operator norm $\lesssim_{p,N} (1+|t|)^{N+1}$. Combined with the Littlewood-Paley inequality (which holds because $(-\Delta)^{(s-z)/2} u \in L^p(\R^N)$), we now have
\begin{equation}%
\begin{split}
\|(-\Delta)^{(s-z)/2} u\|_{L^p(\R^N)}
&\simeq_{p,N} \Big\| \Big( \sum_{j \in \Z} |\Delta_j (-\Delta)^{(s-z)/2} u|^2 \Big)^{1/2} \Big\|_{L^p(\R^N)} \\
&\lesssim_{p,N} (1+|t|)^{N+1} \Big\| \Big( \sum_{j \in \Z} |2^{js} \Delta_j u|^2 \Big)^{1/2} \Big\|_{L^p(\R^N)} \\
&= A(p,N) (1+|t|)^{N+1} \|u\|_{\dot{F}^{s}_{p,2}(\R^N)}
\end{split}
\end{equation}%
the middle inequality following from \eqref{eq:reproduce} and the boundedness of the operator in \eqref{eq:vs} on $L^p(\ell^2)$. This completes the proof of (a).


To deduce (b), one can either appeal to the boundedness of the Riesz transform $\nabla (-\Delta)^{-1/2}$ on $L^p(\R^N)$, or repeat the argument above. We omit the details.
\end{proof}

Furthermore, we will need to consider, for $1 < p < \infty$, the Lorentz space $L^{p,1}(\R^N \times \R^N)$, which is defined to be the set of all measurable functions $g(x,y)$ on $\R^N \times \R^N$ for which
\begin{equation}%
[ g ]_{L^{p,1}(\R^N \times \R^N)} := p \int_0^{\infty} \mathcal{L}^{2N}(\{(x,y) \in \R^N \times \R^N \colon |g(x,y)| \geq \lambda \})^{1/p} d\lambda < \infty.
\end{equation}

Just like $[\cdot]_{L^{p,\infty}}$, the quantity $[ \cdot ]_{L^{p,1}}$ is not a norm because it does not satisfy the triangle inequality; it is only a quasi-norm. Nevertheless, for $1 < p < \infty$, both $L^{p,1}$ and $L^{p,\infty}$ admit a comparable norm, which make them Banach spaces, and $L^{p,\infty}$ is the dual space of $L^{p',1}$ whenever $1/p + 1/p' = 1$: in fact, the easiest way to norm $L^{p',1}$ is to define
\begin{equation}%
\|g\|_{L^{p',1}(\R^N \times \R^N)} := \sup \left\{ \left| \int_{\R^N \times \R^N} g(x,y) G(x,y) dx dy \right| \colon [G]_{L^{p,\infty}(\R^N \times \R^N)} = 1 \right\}.
\end{equation}
If $p \in (1,\infty)$, every $g \in L^{p',1}(\R^N \times \R^N)$ can be approximated in the $L^{p',1}$ norm by functions in $L^{p',1}(\R^N \times \R^N)$ that are compactly supported in the open set $\{(x,y) \in \R^N \times \R^N \colon x \ne y\}$ (because such approximation is possible in the comparable $L^{p',1}(\R^N \times \R^N)$ quasi-norm by the dominated convergence theorem). We are now ready to prove Theorem~\ref{th1.2}.

\begin{proof}[Proof of Theorem~\ref{th1.2}.] We fix $s \in (0,1)$, $p \in (1,\infty)$, $u \in \mathcal{F}^{-1} [C^{\infty}_c(\R^N \setminus \{0\})]$, and $g \in L^{p',1}(\R^N \times \R^N)$ with compact support in $\{(x,y) \in \R^N \times \R^N \colon x \ne y\}$. Consider the function
\begin{equation}%
H(z) = \int_{\R^N \times \R^N} g(x,y) \frac{(-\Delta)^{(s-z)/2}u(x)- (-\Delta)^{(s-z)/2}u(y)}{|x-y|^{\frac{N}{p}+z}} dx dy.
\end{equation}%
This is an entire function of $z$, and we claim that it is a bounded function on the strip $\{z \in \mathbb{C} \colon 0 \leq \text{Re}\, z \leq 1\}$. Indeed, for $u \in \mathcal{F}^{-1} [C^{\infty}_c(\R^N \setminus \{0\})]$, \eqref{eq:Dpow} gives
\begin{equation}%
(-\Delta)^{(s-z)/2} u(x) = \int_{\R^N} (2\pi |\xi|)^{s-z} \widehat{u}(\xi) e^{2\pi i x \cdot \xi} d\xi \quad \text{for all $z \in \mathbb{C}$},
\end{equation}%
so
\begin{equation}%
\begin{split}
| (-\Delta)^{(s-z)/2}u(x)- (-\Delta)^{(s-z)/2}u(y) |
&\leq 2 \|(-\Delta)^{(s-z)/2} u\|_{L^{\infty}(\R^N)} \\
& \leq 2 \int_{\xi \in \textrm{supp} \, \widehat{u}} (2\pi |\xi|)^{s-\textrm{Re}\, z} |\widehat{u}(\xi)| d\xi \lesssim \exp(a_1|\textrm{Re}z|)
\end{split}
\end{equation}%
if $a_1 > 0$ is large enough so that $\max\{2\pi |\xi|,\frac{1}{2\pi |\xi|}\} \leq \exp(a_1)$ for all $\xi \in \textrm{supp} \, \widehat{u}$.
Also, on the support of $g(x,y)$, we have
\begin{equation}%
\left| \frac{1}{|x-y|^{\frac{N}{p}+z}} \right| \lesssim \exp(a_2|\textrm{Re}z|)
\end{equation}%
if $a_2 > 0$ is large enough so that $\max\{|x-y|,|x-y|^{-1} \colon (x,y) \in \text{supp}\, g\} \leq \exp(a_2)$.
Finally, since $g \in L^{p',1}(\R^N \times \R^N)$ has compact support, it is in $L^1(\R^N \times \R^N)$ as well. So
\begin{equation}%
|H(z)| \lesssim \|g\|_{L^1} \exp(a|\textrm{Re}z|) \quad \text{for all $z \in \mathbb{C}$}
\end{equation}%
where $a = a_1 + a_2$. In particular, $H(z)$ is bounded on the strip $\{z \in \mathbb{C} \colon 0 \leq \textrm{Re}\, z \leq 1\}$, as claimed.

Furthermore, for $\text{Re}\, z = 0$, the upper bound in \eqref{eq5'}, together with \eqref{eq:Rez=0}, show that
\begin{equation}%
\begin{split}
|H(z)| &\leq \int_{\R^N \times \R^N} |g(x,y)| \frac{|(-\Delta)^{(s-z)/2}u(x)- (-\Delta)^{(s-z)/2}u(y)|}{|x-y|^{\frac{N}{p}}} dx dy \\
&\leq \|g\|_{L^{p',1}(\R^{2N})} \Big[\frac{(-\Delta)^{(s-z)/2}u(x)- (-\Delta)^{(s-z)/2}u(y)}{|x-y|^{\frac{N}{p}}} \Big]_{L^{p,\infty}(\R^{2N})} \\
&\leq 2 c_2^{1/p} \|g\|_{L^{p',1}(\R^{2N})} \|(-\Delta)^{(s-z)/2} u\|_{L^p(\R^N)} \\
&\leq 2 c_2^{1/p} A (1+|\text{Im}\, z|)^{N+1} \|g\|_{L^{p',1}(\R^{2N})} \|u\|_{\dot{F}^{s}_{p,2}(\R^N)}.
\end{split}
\end{equation}%
On the other hand, for $\text{Re}\, z = 1$, the upper bound in \eqref{eq5}, together with \eqref{eq:Rez=1}, show that
\begin{equation}%
\begin{split}
|H(z)| &\leq \int_{\R^N \times \R^N} |g(x,y)| \frac{|(-\Delta)^{(s-z)/2}u(x)- (-\Delta)^{(s-z)/2}u(y)|}{|x-y|^{\frac{N}{p}+1}} dx dy \\
&\leq \|g\|_{L^{p',1}(\R^{2N})} \Big[ \frac{(-\Delta)^{(s-z)/2}u(x)- (-\Delta)^{(s-z)/2}u(y)}{|x-y|^{\frac{N}{p}+1}} \Big]_{L^{p,\infty}(\R^{2N})} \\
&\leq C^{1/p} \|g\|_{L^{p',1}(\R^{2N})} \|\nabla (-\Delta)^{(s-z)/2} u\|_{L^p(\R^N)} \\
&\leq C^{1/p} A (1+|\text{Im}\, z|)^{N+1} \|g\|_{L^{p',1}(\R^{2N})} \|u\|_{\dot{F}^{s}_{p,2}(\R^N)}.
\end{split}
\end{equation}%
(The upper bound in \eqref{eq5} applies because $(-\Delta)^{(s-z)/2} u \in \mathcal{S}(\R^N) \subset W^{1,p}(\R^N)$ when $u \in \mathcal{F}^{-1} [C^{\infty}_c(\R^N \setminus \{0\})]$, allowing us to invoke \cite[Lemma 3.1]{Po}.)
This allows us to use the three lines lemma from complex analysis to the bounded holomorphic function $H(z)/(z+1)^{N+1}$ on the strip $\{z \in \mathbb{C} \colon 0 \leq \text{Re}\, z \leq 1\}$, and conclude that
\begin{equation}%
|H(s)| \lesssim_{p,N} \|g\|_{L^{p',1}(\R^{2N})} \|u\|_{\dot{F}^{s}_{p,2}(\R^N)}.
\end{equation}%
Taking supremum over $g$, we get
\begin{equation}
\Big[\frac{u(x)-u(y)}{|x-y|^{\frac{N}{p}+s}} \Big]_{L^{p,\infty}(\R^{2N})} \leq C'  \|u\|_{\dot{F}^{s}_{p,2}(\R^N)}
\end{equation}
where $C' = C'(p,N)$, and this inequality holds for all $u \in \mathcal{F}^{-1} [C^{\infty}_c(\R^N \setminus \{0\})]$. This shows that the left-hand side may be defined by density for all $u \in \dot{F}^{s}_{p,2}(\R^N)$, and that the inequality continues to hold after such extension for all $u \in \dot{F}^{s}_{p,2}(\R^N)$.
\end{proof}

\begin{proof}[Proof of Proposition \ref{th1.3}]
We just note that for $u \in \mathcal{F}^{-1} [C^{\infty}_c(\R^N \setminus \{0\})]$, we have
\begin{align}
& \quad \left\| \frac{u(x)-u(y)}{|x-y|^{\frac{N}{p}+s}} \right\|_{L^p(\R^{2N})} = \left\| \frac{u(x+z)-u(x)}{|z|^{\frac{N}{p}+s}} \right\|_{L^p(\R^{2N})} \\
&\lesssim_N \left( \sum_{k \in \Z} 2^{k s p} \sup_{|z| \simeq 2^{-k}} \left\|u(x+z)-u(x)\right\|_{L^p(dx)}^p \right)^{1/p} \\
&\lesssim_{N,p} \left( \sum_{k \in \Z} 2^{k s p} \sup_{|z| \simeq 2^{-k}} \int_{\R^N} \Big( \sum_{j \in \Z} | \Delta_{j+k} u(x+z) - \Delta_{j+k} u(x) |^2 \Big)^{p/2} dx \right)^{1/p},\label{7}
\end{align}
the last inequality following from Littlewood-Paley (note that the sum in $j$ has only finitely many non-zero terms). We consider two cases.

\noindent{\textbf{Case 1}: $1 < p \leq 2$}

In this case, we bound \eqref{7} using the inequality
\begin{equation}
|\sum_j F_j|^{p/2} \leq \sum_j |F_j|^{p/2}
\end{equation}
with $F_j := | \Delta_{j+k} u(x+z) - \Delta_{j+k} u(x) |^2$. We get
\begin{align}
\eqref{7}
&\lesssim \left( \sum_{k \in \Z}  2^{k s p} \sup_{|z| \simeq 2^{-k}} \int_{\R^N} \sum_{j \in \Z} | \Delta_{j+k} u(x+z) - \Delta_{j+k} u(x) |^p dx \right)^{1/p} \\
&\leq \left( \sum_{k \in \Z}  2^{k s p} \sum_{j \in \Z}  \sup_{|z| \simeq 2^{-k}} \| \Delta_{j+k} u(x+z) - \Delta_{j+k} u(x) \|_{L^p(dx)}^p \right)^{1/p}. \label{10}
\end{align}
Now if $|z| \simeq 2^{-k}$, we write
\begin{equation} \label{11}
\Delta_{j+k} u(x+z) - \Delta_{j+k} u(x) = \int_0^1 \frac{d}{dt} \Delta_{j+k} u(x + tz) dt
= \int_0^1 z \cdot \nabla \Delta_{j+k} u(x+tz) dt,
\end{equation}
so its $L^p$ norm with respect to $x$ is bounded by
\begin{equation}%
|z| \|\nabla \Delta_{j+k} u\|_{L^p(\R^N)} \lesssim 2^{-k} \|\nabla \Delta_{j+k} u\|_{L^p(\R^N)}  \lesssim_N 2^j \|\Delta_{j+k} u\|_{L^p(\R^N)}.
\end{equation}%
This shows
\begin{equation}%
\sup_{|z| \simeq 2^{-k}} \|\Delta_{j+k} u(x+z) - \Delta_{j+k} u(x)\|_{L^p(\dif x)} \lesssim_N 2^j\|\Delta_j u\|_{L^p(\R^N)}.
\end{equation}%
We also have the trivial bound
\begin{equation}%
\sup_{|z| \simeq 2^{-k}} \|\Delta_{j+k} u(x+z) - \Delta_{j+k} u(x)\|_{L^p(\dif x)} \leq 2 \|\Delta_j u\|_{L^p(\R^N)}.
\end{equation}%
Then combining these two estimate, we write
\begin{equation} \label{eq:Deltakdiff}
\sup_{|z| \simeq 2^{-k}} \|\Delta_{j+k} u(x+z) - \Delta_{j+k} u(x)\|_{L^p(\dif x)} \lesssim_N
\begin{cases}
\|\Delta_{j+k} u\|_{L^p(\R^N)} &\quad \text{if $j > 0$},\\
2^j \|\Delta_{j+k} u\|_{L^p(\R^N)} &\quad \text{if $j \leq 0$}.
\end{cases}
\end{equation}

Then by substituting \eqref{eq:Deltakdiff} into \eqref{10}, we obtain
\begin{align}
& \quad \left\| \frac{u(x)-u(y)}{|x-y|^{\frac{N}{p}+s}} \right\|_{L^p(\R^{2N})} \\
&\lesssim_N \Big[ \sum_{k \in \Z}  2^{k s p} \Big(  \sum_{j > 0} \|\Delta_{j+k} u\|_{L^p(\R^N)}^p + \sum_{j \leq 0} 2^{j p} \|\Delta_{j+k} u\|_{L^p(\R^N)}^p \Big) \Big]^{1/p} \\
&= \Big[  \sum_{j > 0} 2^{-j s p} \sum_{k \in \Z} 2^{(j+k) s p} \|\Delta_{j+k} u\|_{L^p(\R^N)}^p + \sum_{j \leq 0} 2^{j (1-s) p} \sum_{k \in \Z} 2^{(j+k) s p} \|\Delta_{j+k} u\|_{L^p(\R^N)}^p \Big]^{1/p} \\
&= \Big[  \sum_{j > 0} 2^{-j s p} + \sum_{j \leq 0} 2^{j (1-s) p} \Big]^{1/p} \|u\|_{\dot{F}^{s}_{p,p}(\R^N)} \\
&\simeq  \left[ \frac{1}{s(1-s)} \right]^{1/p} \|u\|_{\dot{F}^{s}_{p,p}(\R^N)}.
\end{align}

\noindent{\textbf{Case 2}: $2 \leq p < \infty$}

We apply Minkowski inequality for $L^{p/2}(\R^N)$ and obtain
\begin{align}
\eqref{7}
& \leq  \left( \sum_{k \in \Z} 2^{k s p} \sup_{|z| \simeq 2^{-k}}\Big( \sum_{j \in \Z}  \| \Delta_{j+k} u(x+z) - \Delta_{j+k} u(x) \|_{L^p(dx)}^2 \Big)^{p/2} \right)^{1/p}\\
& \leq  \left( \sum_{k \in \Z} 2^{k s p} \Big( \sum_{j \in \Z}  \sup_{|z| \simeq 2^{-k}} \| \Delta_{j+k} u(x+z) - \Delta_{j+k} u(x) \|_{L^p(dx)}^2 \Big)^{p/2} \right)^{1/p}. \label{22}
\end{align}
But by \eqref{eq:Deltakdiff}, we have
\begin{align}
\eqref{22}
&\leq \left( \sum_{k \in \Z} 2^{k s p} \Big( \sum_{j > 0}  \| \Delta_{j+k} u \|_{L^p(dx)}^2  + \sum_{j \leq 0} 2^{2 j} \| \Delta_{j+k} u \|_{L^p(dx)}^2 \Big)^{p/2} \right)^{1/p} \\
& = \left\| \sum_{j > 0} 2^{-2 j s} 2^{2 (j+k) s} \| \Delta_{j+k} u \|_{L^p(dx)}^2  + \sum_{j \leq 0} 2^{2 j (1-s)} 2^{2 (j+k) s} \| \Delta_{j+k} u \|_{L^p(dx)}^2 \right\|_{\ell^{p/2}_k}^{1/2}.
\end{align}
Applying the Minkowski inequality for $\ell^{p/2}$, we bound this by
\begin{align}
\Big (\sum_{j > 0} 2^{-2 j s} + \sum_{j \leq 0} 2^{2 j (1-s)} \Big)^{1/2} \|u\|_{\dot{F}^{s}_{p,p}(\R^N)}
\lesssim  \left[ \frac{1}{s(1-s)} \right]^{1/2}  \|u\|_{\dot{F}^{s}_{p,p}(\R^N)}.
\end{align}
\end{proof}

\section{Appendix: Density in Triebel-Lizorkin spaces}

In the proof of Theorem~\ref{th1.2}, we appealed to the case $s \in (0,1)$ and $q = 2$ of the following proposition. Thus we include a sketch of its proof.

\begin{proposition}
$\mathcal{F}^{-1} [C^{\infty}_c(\R^N \setminus \{0\})]$ is dense in $\dot{F}^{s}_{p,q}(\R^N)$ for $s \in \R$, $p \in (1,\infty)$ and $q \in (1,\infty)$.
\end{proposition}

\begin{proof}
Fix $s \in \R$, $p \in (1,\infty)$ and $q \in (1,\infty)$.
First, for $u \in \dot{F}^{s}_{p,q}(\R^N)$, since
\begin{equation}%
\lim_{J \rightarrow +\infty} \Big\| \Big( \sum_{|j| > J} |2^{js} \Delta_j u|^q \Big)^{1/q} \Big\|_{L^p(\R^N)} = 0,
\end{equation}%
we see that $u_J := \sum_{|j| \leq J} \Delta_j u$ converges in $\dot{F}^{s}_{p,q}(\R^N)$ as $J \rightarrow +\infty$.

Next, let $\phi \in \mathcal{S}(\R^N)$ with $\phi(0) = 1$ whose Fourier transform $\widehat\phi$ is compactly supported on the unit ball. For every fixed $J \in \N$, we let $u_{J, \delta}(x) := \phi(\delta x) u_J(x)$. Then for $\delta \ll 2^{-J}$, we have $u_{J,\delta} \in \mathcal{F}^{-1} [C^{\infty}_c(\R^N \setminus \{0\})]$. Thus it remains to show that $u_{J,\delta} \rightarrow u_J$ in $\dot{F}^{s}_{p,q}$ as $\delta \rightarrow 0$. To see this, note that $u_{J,\delta}$ is $C^{\infty}$ on $\R^N$, so $u_{J,\delta}$ converges pointwisely to $u_J$ as $\delta \rightarrow 0$. Furthermore, $u_{J,\delta}$ is dominated by a multiple of $u_J$, which is in $L^p(\R^N)$, so by the dominate convergence theorem,
\begin{equation}%
\lim_{\delta \rightarrow 0} \|u_{J,\delta} - u_J\|_{L^p(\R^N)} = 0.
\end{equation}%
As a result,
\begin{equation}%
\lim_{\delta \rightarrow 0} \|\Delta_j (u_{J,\delta} - u_J)\|_{L^p(\R^N)} = 0
\end{equation}%
for every $j \in \Z$, which implies the desired convergence of $u_{J,\delta}$ to $u_J$ in $\dot{F}^{s}_{p,q}$ as $\delta \rightarrow 0$.
\end{proof}

\end{document}